\documentclass{amsart}
\usepackage{a4wide,amsmath,amssymb}
\usepackage[toc,page]{appendix}

\usepackage{url}
\usepackage{hyperref}
\usepackage{xcolor}

\newtheorem{thm}{Theorem}[section]
\newtheorem{lemma}[thm]{Lemma}
\newtheorem{prop}[thm]{Proposition}
\newtheorem{cor}[thm]{Corollary}
\theoremstyle{remark}
\newtheorem{rem}[thm]{Remark}

\theoremstyle{definition}

\newtheoremstyle{Claim}{}{}{\itshape}{}{\itshape\bfseries}{:}{ }{#1}
\theoremstyle{Claim}

\newcommand{\T}{{\mathbb{T}}}
\newcommand{\Z}{{\mathbb{Z}}}

\newcommand{\R}{\mathbb{R}}

\newcommand{\eps}{\varepsilon}
\newcommand{\norm}[1]{\left\lVert#1\right\rVert}

\theoremstyle{plain}
\def\sideremark#1{\ifvmode\leavevmode\fi\vadjust{% The remark
\vbox to0pt{\hbox to 0pt{\hskip\hsize\hskip1em% will appear only
\vbox{\hsize3cm\tiny\raggedright\pretolerance10000% on the side
\noindent #1\hfill}\hss}\vbox to8pt{\vfil}\vss}}}% in 3cm

\makeatletter
\@namedef{subjclassname@2020}{\textup{2020} Mathematics Subject Classification}
\makeatother

\begin{document}

\title[]{A priori Lipschitz estimates for nonlinear equations with mixed local and nonlocal diffusion via the adjoint-Bernstein method}


%    author one information
%\author{Marco Cirant}
%\address{Dipartimento di Matematica ``Tullio Levi-Civita'', Universit\`a degli Studi di Padova, 
%via Trieste 63, 35121 Padova (Italy)}
%\curraddr{}
%\email{cirant@math.unipd.it}
%\thanks{}
%
%%    author two information
\author{Alessandro Goffi}
\address{Dipartimento di Matematica ``Tullio Levi-Civita'', Universit\`a degli Studi di Padova, 
via Trieste 63, 35121 Padova (Italy)}
\curraddr{}
\email{alessandro.goffi@unipd.it}
%\thanks{}

%\subjclass[2010]{Primary: 35J60, 35B53, 35D40}
%%    The 2010 edition of the Mathematics Subject Classification is
%%    now available.  If you are citing a classification from the
%%    new scheme, use the following input coding instead.
 \thanks{%Partially supported by ....
 The author is member of the Gruppo Nazionale per l'Analisi Matematica, la Probabilit\`a e le loro Applicazioni (GNAMPA) of the Istituto Nazionale di Alta Matematica (INdAM). The author was partially supported by the INdAM-GNAMPA Project 2022 ``Propriet\`a quantitative e qualitative per EDP non lineari con termini di gradiente'' and by the King Abdullah University of Science and Technology (KAUST) project CRG2021-4674 ``Mean-Field Games: models, theory and computational aspects". The author wishes to thank Prof. G. Barles for discussions on gradient estimates via viscosity solutions' methods and for providing many references on the subject.}
% % The authors were partially supported by the research project  ``Nonlinear Partial Differential Equations: Asymptotic Problems and Mean-Field Games" of the Fondazione CaRiPaRo. 
% }

\date{\today}

\begin{abstract}
We establish a priori Lipschitz estimates for equations with mixed local and nonlocal diffusion, coercive gradient terms and unbounded right-hand side in Lebesgue spaces through an integral refinement of the Bernstein method. This relies on a nonlinear, nonlocal and variational version of the Bochner identity that involves the adjoint equation of the linearization of the initial problem.
\end{abstract}

\maketitle


\section{Introduction}
The Bernstein method \cite{Bernstein} is nowadays a classical technique to obtain gradient estimates for second order elliptic equations \cite{GT}. The core idea behind this approach is extremely simple and relies on the so called Bochner identity for the Laplacian on the Euclidean space $\R^N$
\[
\Delta|\nabla u|^2=2|D^2u|^2+2\nabla u\cdot \nabla \Delta u.
\]
This in particular shows that if $u$ is a solution to an elliptic/parabolic equation on some domain of the Euclidean space, then $w=|\nabla u|^2$ is a subsolution to an elliptic/parabolic equation. As a consequence, the gradient bound follows from the maximum principle, whenever one knows a priori that $\nabla u$ is bounded on the boundary of the domain, which is typically a consequence of the existence of barrier-like functions, cf \cite{GT}. This nonvariational technique to obtain global (and even local) bounds based on the maximum principle has been fruitfully extended to many nonlinear PDEs, such as quasilinear elliptic equations \cite{GT}, fully nonlinear second order equations \cite{CC}, semilinear equations (with superlinear gradient growth), even driven by the $p$-Laplacian, see e.g. \cite{LionsBook,SerrinProc} and the more recent work \cite{PorrCCM}, integro-differential problems, both linear and (fully) nonlinear, see \cite{CDV}. This approach is also the cornerstone to deduce quite different qualitative and quantitative properties for elliptic and parabolic equations, such as differential Harnack estimates and Liouville theorems, see \cite{SerrinPeletier,Lions85,GarofaloFrac}.\\
One of the main drawbacks of the standard Bernstein method relies on the regularity requirements necessary to carry out the computations. Since third derivatives appear in the Bochner identity, $u$ needs to be smooth enough (e.g. of class $C^3$), and hence only a priori estimates can be derived. Then one has to find a suitable regularization/approximation of the equation having smooth enough solutions to really obtain the regularity estimate after passing to the limit, see e.g. \cite[Remark 3]{CGell}, or instead work at the level of difference quotients starting with suitable weak solutions, cf \cite{cg20}.\\
This difficulty was partially circumvented through the introduction of the so-called weak Bernstein method introduced by G. Barles \cite{Barles91,Barles21} in the realm of fully nonlinear equations, which consists in shifting the attention, after a change of variable, to the maximum of the function
\[
(x,y)\longmapsto u(x)-u(y)-L|x-y|\ ,(x,y)\in\overline{\Omega}\times\overline{\Omega},
\]
$\Omega$ being the state space. The idea is to prove that if it is achieved when $x=y$ for $L$ large enough, then $|\nabla u|\leq L$. The previous idea basically corresponds to look at the equation satisfied by $|\nabla u|^2$, and the structure conditions are similar to those required to run the classical Bernstein argument for nonlinear elliptic equations. There are, among others, three peculiar features of this method: it does not require regular solutions (in particular it applies to viscosity solutions), it does not need strong ellipticity, being applicable to problems with fractional or degenerate diffusion \cite{CDLP,BLT17}, and it allows to treat gradient terms with arbitrary growth. Nonetheless, it still requires $f\in W^{1,\infty}$, as in the standard Bernstein argument.\\
Another approach in the framework of viscosity solutions based again on a slightly different doubling variables method has been introduced by H. Ishii and P.-L. Lions \cite{IshiiLions} to obtain $C^{\alpha}$, $0<\alpha\leq1$, estimates: this method takes advantage of the ellipticity of the diffusion to control the coercivity of the gradient term, being particularly designed for problems with first-order terms below the natural growth and H\"older or bounded coefficients, cf Assumption (3.16) in \cite{IshiiLions}, see e.g. \cite{BCCI12,LeyNguyen} and also \cite{PP,DirrNguyen} for further developments. A general reference discussing both these procedures to get gradient bounds in the context of viscosity solutions is the paper by G. Barles and P. Souganidis \cite{BS}. \\
Viscosity solutions' techniques have been refined for the application to problems with superquadra\-tic gradient growth in \cite{CDLP}. Nevertheless, both these methods have a drawback in terms of the regularity of the data. Being based on the notion of viscosity solution, they rely on the maximum principle and they require at least continuous or Lipschitz data. Albeit the notion of viscosity solution admits variants that allow to encompass discontinuous Hamiltonians with merely summable data (at least when dealing with local terms), no technique seems available to derive gradient bounds for problems with fractional diffusion, strongly coercive gradient terms (especially in the supernatural regime) and data in $L^q$. \\
To this aim, different methods have been explored when the equation has ``ingredients'' belonging to Lebesgue spaces. Such techniques, usually named integral Bernstein methods, are again based on delicate integral refinements of the Bochner identity and started with the work by P.-L. Lions \cite{Lions85}. Equivalently, they can be formulated at the level of the variational formulation of the equation, choosing a $p$-Laplacian of suitable order as a test function. These integral approaches have been extended more recently in \cite{CGell} (see also the references therein), in the study of nonlinear Calder\'on-Zygmund estimates for elliptic problems with superlinear gradient terms. They have been used even in connection with $p$-Laplacian problems without the presence of first-order terms  \cite{CianchiMazyaCPDE,PorrCCM}. Most of these integral techniques for nonlinear equations with power-growth terms have in common the use of the so-called Bakry-\'Emery curvature dimensional inequality \cite{BGL}, that in $\R^N$ reads as
\begin{equation}\label{CD}
|D^2u|^2\geq \frac{1}{N}(\Delta u)^2,
\end{equation}
and it is a consequence of the Cauchy-Schwarz inequality, cf \cite[Definition 20.7]{GarofaloFrac}. This inequality is, among others, a powerful tool to obtain logarithmic gradient estimates for positive harmonic functions, and hence the Harnack inequality. In the nonlinear setting, it is crucial to handle nonlinearities with superlinear gradient growth, at least for elliptic problems. Indeed, if $u$ solves $-\Delta u+|\nabla u|^\gamma=f(x)$, $\gamma>1$, then \eqref{CD} combined with the algebraic inequality $(a-b)^2\geq \frac{a^2}{2}-2b^2$, $a,b\in\R$, imply
\[
|D^2u|^2\geq \frac{|\nabla u|^{2\gamma}}{2N}-\frac{2}{N}f^2,
\]
which allows to gain an additional degree of coercivity respect to $|\nabla u|^2$ through the term $|\nabla u|^{2\gamma}$ and conclude the gradient bound. This crucial and deep step has (once more) a drawback: it does not apply neither to equations involving fractional operators (even in the stationary case) nor to evolution equations with time-dependent source terms belonging to Lebesgue spaces. Indeed, as for the latter, if $u$ solves $\partial_t u-\Delta u+|\nabla u|^\gamma=f(x,t)$, then \eqref{CD} implies the presence of a term involving a time-derivative term that can be absorbed only when one knows a priori that $\partial_tu \geq -C$. For instance, this is the case when $f=f(x)$ or even when $\partial_tf(x,t)$ is essentially bounded or at least belongs to some Lebesgue space. \\
The case of the presence of a fractional diffusion is even worse. In this setting the fractional Bochner identity, cf \cite[equation (2.10)]{Constantin}, replaces $|D^2u|^2$ with the nonlocal term
\[
\int\frac{|\nabla u(x,t)-\nabla u(x+y,t)|^2}{|y|^{N+2s}}\,dy,
\]
which however does not allow to deduce a fractional version of \eqref{CD}. Actually, whether an inequality like \eqref{CD} (for a possibly different constant) holds for the fractional Laplacian was raised in \cite[equation (20.14)]{GarofaloFrac}, and recently answered negatively in \cite{ZacherCPDE}, showing thus that some new technique with respect to the ``classical'' integral Bernstein approach is needed to obtain gradient bounds for fractional equations with merely integrable data. \\

The aim of this note is thus to propose a new Bernstein-type argument to prove Lipschitz estimates that avoids the use of \eqref{CD}, namely the gain of (additional) coercivity by plugging the equation, to deduce gradient bounds in the aforementioned ``negative'' situations, i.e. when the equation is parabolic and/or it presents a nonlocal diffusive term with also unbounded terms in Lebesgue spaces. This would imply existence and uniqueness of solutions as a byproduct through the contraction mapping principle combined with a continuation argument, see \cite{CGsima}.\\
We will focus on the Cauchy problem (for simplicity posed on the $N$-dimensional flat torus)
\begin{equation}\label{eqmain}
\begin{cases}
\partial_tu+\mathcal{L}u+H(x,\nabla u)=f(x,t)&\text{ in }Q_T:=\T^N\times(0,T),\\
u(x,0)=u_0(x)&\text{ in }\T^N\ ,
\end{cases}
\end{equation}
where $\mathcal{L}$ is a diffusion operator defined by
\[
\mathcal{L}:=-\eps\Delta+\mu(-\Delta)^s
\]
when $\eps,\mu>0$. Such mixed diffusive operators have received an increasing attention during the last years, see e.g. the series of works by S. Biagi-S. Dipierro-E. Valdinoci-E. Vecchi \cite{BDVVprse,BDVVcpde} and the recent one by C. De Filippis-G. Mingione \cite{DFM}. The achievement of $L^\infty$ gradient bounds will rely heavily on the analysis of the regularity of solutions to the dual equation of the linearization of \eqref{eqmain}, and it is largely inspired by a method introduced by L.C. Evans \cite{Evans2} to study gradient shock structures of first-order Hamilton-Jacobi equations with non-convex Hamiltonians. Here, we consider
\begin{equation}\label{eqrho}
\begin{cases}
-\partial_t\rho+\mathcal{L}^*\rho-\mathrm{div}(H_p(x,\nabla u)\rho)=0&\text{ in }Q_\tau:=\T^N\times(0,\tau),\\
\rho(x,\tau)=\rho_\tau(x)&\text{ in }\T^N\ ,
\end{cases}
\end{equation}
where $H_p$ stands for the derivative of $H$ with respect to the second entry and
\[
\mathcal{L}^*:=-\eps\Delta+\mu(-\Delta)^s.
\]
We emphasize that such an approach has been recently implemented for viscous Hamilton-Jacobi equations with coercive gradient terms in \cite{cg20}, from which we borrow most of the ideas used in this note. In particular, we exploit the regularizing effect of both the local and the first-order term, therefore can be considered of nonperturbative type, whilst the nonlocal one acts only as a perturbation. The main idea to derive the gradient bound will be based on a variational nonlocal version of a (nonlinear) Bochner identity. Indeed, if $u$ solves \eqref{hjm}, straightforward computations lead to the following identity satisfied by $w=\frac12|\nabla u|^2$:
\begin{multline*}
\partial_t w(x,t)+\mathcal{L}w+\eps|D^2u(x,t)|^2+\frac{\mu}{2}\int_{\T^N}|\nabla u(x,t)-\nabla u(x+y,t)|^2K(y)\,dy+H_p(x,\nabla u(x,t))\cdot \nabla w(x,t)\\
+H_x(x,\nabla u(x,t))\cdot \nabla u(x,t)=\nabla f(x,t)\cdot\nabla u(x,t),
\end{multline*}
where $K$ is the kernel of the fractional Laplacian on the torus, cf \cite{RS}. Then, the central point is to test the previous identity for $w$ against the function $\rho$ solving the dual problem \eqref{eqrho}, see Lemma \ref{bocdual}, and handle all the integral terms through a delicate interplay between integrability estimates for transport equations, Sobolev and Young's inequalities. Similar duality methods were previously used in various contexts for local problems, see e.g. \cite{Gomesbook,TranBook}, and even to deduce semiconcavity and Lipschitz bounds for equations with mixed diffusion and regular data in \cite{CGsima} in the context of parabolic fractional Mean Field Games. This procedure applies, with suitable modifications, even to stationary problems of the form
\[
\mathcal{L}u+u+H(x,\nabla u)=f(x)\text{ in }\T^N
\]
through the study of the regularity properties of its dual counterpart
\[
\mathcal{L}^*\rho+\rho-\mathrm{div}(H_p(x,\nabla u)\rho)=0\text{ in }\T^N,
\]
see e.g. \cite{Goffi21dir}. Moreover, the technique of the present paper applies with few modifications to problems driven by the more general local-nonlocal operator
\[
\mathcal{L}_{A,b,c,s}u=-\mathrm{Tr}(A(x,t)D^2u)+b(x,t)\cdot Du+c(x,t)u+(-\Delta)^su,
\]
under suitable regularity assumptions on the coefficients, and even to equations with more general integro-differential operators for which a duality theory holds, cf \cite{GarroniMenaldi2}. We emphasize once more that, though  one expects the same results of the case of a local diffusion, being the dominating part of the diffusive term, obtaining a gradient estimate through the standard Bernstein method is by no means immediate even for the stationary problem involving the mixed operator.\\

We mention that refinements of the Bernstein technique in the nonlocal setting have been proposed quite recently in the manuscript \cite{CDV}, which develops the Bernstein method for integro-differential equations (even fully nonlinear) without lower-order terms, and also in \cite{Constantin,CTV15} to study gradient bounds for solutions to some different nonlocal models. Nonetheless, we believe that such formulations do not allow to treat problems with power-growth nonlinearities and $L^q$ data as those appearing in the PDEs of the present paper, being of nonvariational nature. We also mention the possible application of our techniques to study regularity properties for some nonlocal problems, where a gradient nonlinearity with polynomial growth appears, such as those arising in combustion theory \cite{Ibdah} driven by the operator $-\Delta-(-\Delta)^s$, that will be the matter of future research.\\

\textit{Plan of the paper}. Section \ref{sec;main} is devoted to state the assumptions and the main results of the paper. Section \ref{sec;prel} provides some preliminary algebraic identities to implement the Bernstein argument. Section \ref{sec;transport} concerns regularity properties of transport diffusion equations with general velocity field $b$ driven by the mixed operator $-\Delta+(-\Delta)^s$. Section \ref{sec;esthj} contains some estimates for solutions to \eqref{hjm} via the results in Section \ref{sec;transport}. Section \ref{sec;proofs} is devoted to the proof of the main result.
\section{Main result}\label{sec;main}
Throughout this manuscript, the state space will be $\T^N$, the $N$-dimensional flat torus. We denote by $L^p(\T^N)$ the space of all measurable and periodic functions on $\R^N$ belonging to $L^p_{\mathrm{loc}}(\R^N)$, endowed with the norm $\|\cdot\|_p=\|\cdot\|_{L^p((0,1)^N)}$. For positive $\mu\in\R$, $W^{\mu,p}(\T^N)$ is the standard fractional Sobolev space of functions on the flat torus, while $H_p^\mu(\T^N)$ denotes the Bessel potential space, i.e. the space of distributions $u$ such that $(I-\Delta)^\frac{\mu}{2}u\in L^p(\T^N)$.
For any time interval $(0,t) \subseteq \R$, let $Q_{t}:=\T^N\times (0, t)$. For any $p\geq1$, we denote by $W^{2,1}_p(Q_t)$ the space of functions $u$ such that $\partial_t^{r}D^{\beta}_xu\in L^p(Q)$ for all multi-indices $\beta$ and $r$ such that $|\beta|+2r\leq  2$, endowed with the norm
\begin{equation*}
\|u\|_{W^{2,1}_p(Q_t)}=\left(\iint_{Q_t}\sum_{|\beta|+2r\leq2}|\partial_t^{r}D^{\beta}_x u|^pdxdt\right)^{\frac1p}.
\end{equation*}
Similarly, the space $W^{1,0}_p(Q_t)$ is equipped with the norm
\[
\norm{u}_{W^{1,0}_p(Q_t)}:=\|u\|_{L^p(Q_t)}+\sum_{|\beta|=1}\|D_x^{\beta}u\|_{L^p(Q_t)}\ .
\]
We define the space $\mathcal{H}_p^{1}(Q_t)$ as the space of functions $u\in W^{1,0}_p(Q_t)$ with $\partial_tu\in (W^{1,0}_{p'}(Q))'$ and norm
\begin{equation*}
\|u\|_{\mathcal{H}_p^{1}(Q_t)}:=\norm{u}_{W^{1,0}_p(Q_t)}+\|\partial_tu\|_{(W^{1,0}_{p'}(Q_t))'}\ .
\end{equation*}
We consider the following Cauchy problem
\begin{equation}\label{hjm}
\begin{cases}
\partial_tu+\mathcal{L}u+H(x,\nabla u)=f(x,t)&\text{ in }Q_T,\\
u(x,0)=u_0(x)&\text{ in }\T^N\ .
\end{cases}
\end{equation}
Here, without loss of generality, we suppose $\eps,\mu=1$ and consider the diffusion operator
\[
\mathcal{L}:=-\Delta+(-\Delta)^s\ ,
\]
\[
s\in(0,1).
\]
We will also denote by $H_x$ and $H_p$ the derivatives of $H$ with respect to the first and second entry respectively. Moreover, we assume that $H\in C^1(\T^N\times\R^N)$ is convex and satisfies
\begin{equation}\label{H}\tag{H}
\begin{aligned}
& \text{there exist constants $\gamma > 1$ and $C_H>0$ such that}\\
& \qquad C_H^{-1}|p|^{\gamma}-C_H\leq H(x,p)\leq C_H(|p|^{\gamma}+1)\ ,\\
& \qquad H_p(x,p)\cdot p-H(x,p)\geq C^{-1}_H|p|^{\gamma}-C_H\ , \\
& \qquad |H_x(x,p)|\leq C_H(|p|^{\gamma}+1) \ , \\
& \qquad C^{-1}_H|p|^{\gamma-1}- {C}_H \le |H_p(x,p)|\leq C_H|p|^{\gamma-1}+{C}_H \ ,
\end{aligned}
\end{equation}
and $f$ will be some source term controlled in some space-time Lebesgue class $L^q(Q_T)$, for some suitable $q>1$. Our main result shows the preservation of the Lipschitz regularity in the equation via a quantitative bound.
\begin{thm}\label{main1}
Suppose that 
\begin{itemize}
\item[(i)] $s\in(0,1)$;
\item[(ii)] $H\in C^1(\T^N\times\R^N)$ and satisfies \eqref{H};
\item[(iii)] $f\in C^1(Q_T)$;
\item[(iii)] $u_0\in W^{1,\infty}(\T^N)$.
\end{itemize}
Let 
\begin{equation}\label{assq}
q>
\begin{cases}
N+2&\text{ if }1<\gamma\leq 3\\
\frac{(N+2)(\gamma-1)}{2}&\text{ if }\gamma>3\ .
\end{cases}
\end{equation}
Then, there exists a constant $C$ depending on $q,N,T,s,C_H,\|u_0\|_{W^{1,\infty}(\T^N)},\|f\|_{L^q(Q_T)}$ such that every smooth solution to \eqref{hjm} satisfies
\[
\|u(\cdot,t)\|_{W^{1,\infty}(\T^N)}\leq C,\ t\in[0,T].
\]
\end{thm}
Let us stress that, though we require $f\in C^1(Q_T)$, the gradient bound depends only on the summability of $f\in L^q_{x,t}$, so it can be regarded as an a priori estimate. In particular, one can avoid to impose $f\in C^1(Q_T)$ implementing a scheme for strong solutions belonging to $W^{2,1}_q(Q_T)=\{\partial_t u,u,\nabla u,D^2u\in L^q(Q_T)\}$: this can be done using a test function argument as detailed in Remark \ref{test}, cf \cite{CGell}, or using a procedure through difference quotients \cite{cg20}, which in a sense avoid the differentiation of the equation.
It is worth remarking that when the source term is essentially bounded, i.e. $f\in L^\infty(Q_T)$, the condition on the summability \eqref{assq} holds, and the results appear to be new even in this framework. Indeed, as already discussed, Lipschitz bounds from the theory of viscosity solutions require the right-hand side of the equation to be at least continuous and time-independent or even more regular (e.g. Lipschitz), so that an estimate on $\partial_t u$ readily follows by the maximum principle. We finally emphasize that a regularity estimate starting from a continuous initial datum and suitable weak solutions can be obtained working at the level of difference quotients, as already done first in \cite{cg20} for Lipschitz regularity and then in \cite{CGpar,Goffi21JEE} for H\"older regularization properties.\\
%As usual, a priori estimates are the lynchpin for the obtainment of a well-posedness result, as stated in the following:
%\begin{thm}\label{main2}
%Let $f\in L^q(Q_T)$ with $q>1$ satisfying \eqref{assq}. Suppose also that $s\in(0,1)$, $H\in C^1(\T^N\times(0,T))$ satisfies \eqref{H} and $u_0\in W^{2-\frac{2}{q},q}(\T^N)$. Then, there exists a unique strong solution $u\in W^{2,1}_q(Q_T)$ of \eqref{hjm}. Furthermore, such a solution belongs to $C^{1+\alpha,\frac{1+\alpha}{2}}(Q_T)$, the space-time H\"older space of order $1+\alpha$ in $x$ and $\frac{1+\alpha}{2}$ in $t$.
%\end{thm}
We conclude by saying that this nonlinear duality method readily leads to a new proof of the gradient bound for strong solutions to an equation without the local diffusion term (i.e. with $\eps=0$) when one imposes $f\in L^\infty(0,T;W^{1,\infty}(\T^N))$, which is a classical assumption in the theory of viscosity solutions. This provides an alternative method of proof of the gradient bounds in \cite[Section 3]{BLT17} (although in a stronger framework than the viscosity one and for viscous problems). Still, even when $f\in L^q(0,T;W^{1,q}(\T^N))$ the results would be new, as the assumptions on the right-hand side can be considered as intermediate between those typically assumed to implement the weak Bernstein argument and the Ishii-Lions method, see Remark \ref{rem;higher}.
\section{Preliminary results}\label{sec;prel}
For $s\in(0,1)$ and $u\in C^\infty(\T^N)$ we recall that the following pointwise identity holds (cf \cite{RS})
\[
(-\Delta)^s u=\text{P.V.}\int_{\T^N}(u(x)-u(x+y))K(y)\,dy,
\]
where 
\[
K(y)=c_{N,s}\sum_{k\in\Z^N}\frac{1}{|x-2\pi k|^{N+2s}}, x\in\T^N\ ,x\neq0
\]
and $c_{N,s}$ is a normalizing positive constant depending on $N,s$, see \cite{RS} for the definition. The main result needed to implement our Bernstein-type estimate is the following:
\begin{prop}\label{eqw}
If $u$ is a smooth solution to \eqref{hjm}, then the evolution of $w=\frac12|\nabla u(x,t)|^2$ is described by the equation
\begin{multline}\label{eqwid}
\partial_t w(x,t)-\Delta w+(-\Delta)^sw+|D^2u(x,t)|^2+\frac{1}{2}\int_{\Omega}|\nabla u(x,t)-\nabla u(x+y,t)|^2K(y)\,dy\\
+H_p(x,\nabla u(x,t))\cdot \nabla w(x,t)+H_x(x,\nabla u(x,t))\cdot \nabla u(x,t)=\nabla f(x,t)\cdot\nabla u(x,t)\text{ in }\Omega\times(0,T)
\end{multline}
equipped with the initial condition $w(x,0)=\frac12|\nabla u(x,0)|^2$, where $\Omega=\T^N$ or $\R^N$ and $|D^2u(x,t)|^2=\sum_{i,j=1}^N(\partial_{x_ix_j}u)^2$.
\end{prop}
Before proving this, we need the following Bochner (pointwise) identity for the mixed local-nonlocal operator $-\Delta+(-\Delta)^s$. This extends an identity already pointed out in \cite[equation (2.10)]{Constantin}.
\begin{lemma}
Let $w(x,t)=\frac{1}{2}|\nabla u(x,t)|^2$. Then $w$ satisfies the identity
\begin{multline}\label{bocmix}
\Delta w(x,t)-(-\Delta)^s w(x,t)=\nabla u(x,t)\cdot \nabla (\Delta u(x,t))+\nabla u(x,t)\cdot \nabla (-(-\Delta)^s u(x,t))\\
+|D^2u(x,t)|^2+\frac{1}{2}\int_{\Omega}|\nabla u(x,t)-\nabla u(x+y,t)|^2K(y)\,dy,
\end{multline}
where $\Omega=\T^N$ or $\R^N$. 
\end{lemma}
\begin{proof}
Standard algebraic computations give
\[
\partial_{x_j}w=\sum_{i=1}^N\partial_{x_i}u\partial_{x_ix_j}u\ ;\partial_{x_jx_j}w=\sum_{i=1}^N[(\partial_{x_ix_j}u)^2+\partial_{x_i}u\partial_{x_ix_jx_j}u],
\]
so, summing over $j$ we have
\[
\Delta w(x,t)=|D^2u(x,t)|^2+\nabla u(x,t)\cdot \nabla \Delta u(x,t).
\]
Moreover, from \cite[Proposition 2.1]{CTV15} or \cite[Lemma 20.2]{GarofaloFrac} for a sufficiently smooth function $v$, we have
\[
\frac12\Delta^s(v^2(x,t))=v(x,t)\Delta^s v(x,t)+\frac{1}{2}\int_{\Omega}(v(x,t)-v(x+y,t))^2K(y)\,dy.
\]
Applying the above identity to any directional derivative $v=\partial_e u$, we get the conclusion after summing the obtained expressions.
\end{proof}
\begin{proof}[Proof of Proposition \ref{eqw}]
It is enough to differentiate \eqref{hjm} with respect to $x_j$, multiply the resulting equation by $\partial_{x_j}u$, taking the sums for $j=1,...,N$ and finally use \eqref{bocmix}.
\end{proof}
We end this section with a maximal $L^q$-regularity property for the heat equation with mixed diffusion.
\begin{lemma}\label{czmix}
Let $q>1$, $V\in L^q(Q_T)$ and $u_0\in W^{2-\frac{2}{q},q}(Q_T)$. Then, there exists a unique strong solution $u\in W^{2,1}_q(Q_T)$ of the evolution problem
\begin{equation*}
\begin{cases}
\partial_tw-\Delta w+(-\Delta)^sw=V(x,t)&\text{ in }Q_T:=\T^N\times(0,T),\\
w(x,0)=w_0(x)&\text{ in }\T^N\ .
\end{cases}
\end{equation*}
Moreover, the following estimate holds
\[
\|w\|_{W^{2,1}_q(Q_T)}\leq C(\|V\|_{L^q(Q_T)}+\|w_0\|_{W^{2-\frac{2}{q},q}(Q_T)}),
\]
where $C$ depends on $s,T,q,N$ and remains bounded for bounded values of $T$.
\end{lemma}
\begin{proof}
We use a contraction mapping argument on the space $\mathcal{C}:=\{u\in W^{2,1}_q(Q_T):u(0)=u_0\}$, following \cite[Theorem 3.7]{GarroniMenaldi2}. For fixed $z\in W^{2,1}_q(Q_T)$, we consider the map $\Psi$ that sends $z$ into the solution of the problem
\[
\begin{cases}
\partial_tw-\Delta w=V(x,t)-(-\Delta)^sz&\text{ in }Q_T:=\T^N\times(0,T),\\
w(x,0)=w_0(x)&\text{ in }\T^N\ .
\end{cases}
\]
Applying interpolation inequalities, see e.g. \cite[Lemma 2.4]{CGsima}, we have for $\delta>0$
\[
\|(-\Delta)^sz\|_{L^q(Q_T)}\leq\delta \|z\|_{W^{2,1}_q(Q_T)}+C(\delta)\|z\|_{L^q(Q_T)}, 
\]
where $C(\delta)$ grows as $\delta\to0$. We now write
\[
z(\cdot,s)=w_0(\cdot)+\int_0^s\partial_t z(\cdot,\omega)\,d\omega
\]
and get by the H\"older's inequality
\[
\|z\|_{L^q(Q_T)}\leq \tau^\frac1q\|w_0\|_{L^q(\T^N)}+\tau\|\partial_tz\|_{L^q(Q_T)}.
\]
Applying maximal $L^q$-regularity for the heat equation with frozen right-hand side we get
\[
\|w\|_{W^{2,1}_q(Q_T)}\leq C(\|V\|_{L^q(Q_T)}+\|(-\Delta)^sz\|_{L^q(Q_T)}+\|w_0\|_{W^{2-\frac{2}{q},q}(Q_T)}),
\]
and then one finds that $\Psi$ is a contraction on $\mathcal{C}$, as done in \cite{GarroniMenaldi2}, by taking first $\delta$ small and then $T\leq T_0$ small enough. Applying the same procedure a finite number of steps, one proves the result for any fixed $T$.
\end{proof}
\section{A priori estimates for general transport equations with mixed local and nonlocal diffusion}\label{sec;transport}
In this section we focus on the following backward problem driven by a general vector field $b=b(x,t)$
\begin{equation}\label{fpm}
\begin{cases}
-\partial_t\rho-\Delta \rho+(-\Delta)^s \rho+\mathrm{div}(b(x,t)\rho)=0&\text{ in }Q_\tau,\\
\rho(x,\tau)=\rho_\tau(x)&\text{ in }\T^N\ .
\end{cases}
\end{equation}
We will also assume that $\rho_\tau\in C^\infty(\T^N)$, $\rho_\tau\geq0$ and $\|\rho_\tau\|_{L^1(\T^N)}=1$. We mainly describe the a priori estimates needed to run our Bernstein argument, without discussing the existence and uniqueness of (weak) solutions to such problem, which however is well-known due to the presence of the heat operator, even when $b$ belongs to the Aronson-Serrin interpolated condition $b\in L^r_t(L^z_x)$, $\frac{N}{2r}+\frac{1}{z}\leq\frac12$, see e.g. \cite[Remark 3.7]{GarroniMenaldi2}. Thus, from now on, we will consider classical solutions, even though the argument can be made rigorous for weak energy solutions \cite{cg20}. \\
The next is a maximal regularity estimate for solutions to \eqref{fpm} in Lebesgue spaces obtained in terms of terminal data belonging to $L^1(\T^N)$.
\begin{prop}\label{prop1}
Let $\rho$ be the nonnegative solution to \eqref{fpm} and let
\[
1<q'<\frac{N+2}{N+1}\ .
\]
Then there exists $C>0$ depending on $\sigma',N,s,T$ such that
\[
\|\rho\|_{\mathcal{H}_{q'}^1(Q_\tau)}\leq C(\|b\rho\|_{L^{q'}(Q_\tau)}+\|\rho\|_{L^{q'}(Q_\tau)}+\|\rho_\tau\|_{L^1(\T^N)})\ .
\]
\end{prop}
\begin{proof}
The proof can be done following either the duality arguments in \cite[Proposition 2.4]{cg20} or regarding the problem \eqref{fpm} as
\begin{equation*}
\begin{cases}
\partial_t\rho-\Delta \rho+(-\Delta)^s \rho=-\mathrm{div}(b(x,t)\rho)&\text{ in }Q_\tau:=\T^N\times(0,\tau),\\
\rho(x,\tau)=\rho_\tau(x)&\text{ in }\T^N\ ,
\end{cases}
\end{equation*}
and then using maximal regularity properties for the linear evolution equation $\partial_t\rho-\Delta \rho+(-\Delta)^s \rho=V(x,t)\in L^q_{x,t}$ as in \cite{Goffi21JEE}, together with the embeddings of (trace) fractional Sobolev spaces. Maximal regularity for the heat equation with mixed diffusion in Lebesgue spaces holds by Lemma \ref{czmix}.
\end{proof}
As a consequence, we have the following
\begin{cor}\label{corfpmain}
Let $\rho$ be the nonnegative solution to \eqref{fpm} and let
\[
1<q'<\frac{N+2}{N+1}\ .
\]
Then there exists $C>0$ depending on $\sigma',N,s,T$ such that
\[
\|\rho\|_{\mathcal{H}_{q'}^1(Q_\tau)}\leq C\left(\iint_{Q_\tau}|b|^{m'}\rho\,dxdt+1\right)\ .
\]
where
\[
m'=1+\frac{N+2}{q}.
\]
\end{cor}
\begin{proof}
The proof is a consequence of the application of the Young's inequality in the estimate of Proposition \ref{prop1}, and it is essentially the same of \cite[Proposition 2.5]{cg20}, so we omit it. 
\end{proof}
\section{A priori estimates for the Hamilton-Jacobi equation by duality and some consequences}\label{sec;esthj}
The main goal of this section is to analyze the following transport equation
\begin{equation}\label{fpms}
\begin{cases}
\partial_t\rho-\Delta \rho+(-\Delta)^s \rho-\mathrm{div}(H_p(x,\nabla u(x,t))\rho)=0&\text{ in }Q_\tau:=\T^N\times(0,\tau),\\
\rho(x,\tau)=\rho_\tau(x)&\text{ in }\T^N\ ,
\end{cases}
\end{equation}
for $\tau\in(0,T)$, $\rho_\tau\in C^\infty(\T^N)$, $\rho_\tau\geq0$ and $\|\rho_\tau\|_{L^1(\T^N)}=1$. Note that by the standing assumptions, $\rho$ is a.e. nonnegative on the cylinder. Notice also that it is the adjoint equation of the linearization of \eqref{hjm}. From now on, we denote by $L:\T^N\times\R^N\to \R$ the Legendre transform of $H$ with respect to the second entry, i.e. 
\[
L(x,\nu)=\sup_{p\in\R^N}\{p\cdot \nu-H(x,p)\}.
\]
By the convexity of $H(x,\cdot)$ it follows that
\[
H(x,p)=\sup_{\nu\in\R^N}\{p\cdot \nu-L(x,\nu)\}
\]
and
\[
H(x,p)=p\cdot \nu-L(x,\nu)\iff \nu=H_p(x,p).
\]
We further recall the following properties of the Lagrangian function $L$ valid for all $\nu\in\R^N$:
\begin{equation}\label{L1}\tag{L1}
C_L^{-1}|\nu|^{\gamma'}-C_L\leq |L(x,\nu)|\leq C_L|\nu|^{\gamma'}.
\end{equation}
%\begin{equation}\label{L2}\tag{L2}
%|L_x(x,\nu)|\leq C_L(|\nu|^{\gamma'}+1).
%\end{equation}
\begin{thm}\label{mainberndual}
Let $u,\rho$ be classical solutions to \eqref{hjm} and \eqref{fpms} respectively, and assume
\begin{equation}\label{boundq}
q>\max\left\{\frac{(N+2)(\gamma-1)}{2},N+2\right\}.
\end{equation}
Then, there exists $C>0$ and $\theta\in(0,1)$ such that
\[
\|\rho\|_{\mathcal{H}_{q'}^1(Q_\tau)}\leq C\left(\|\nabla u\|_{L^\infty(Q_\tau)}^{1-\theta}+1\right),
\]
where $C$ depends on $C_H,s,q,d,T,\|u_0\|_{C(\T^N)},\|f\|_{L^q(Q_T)}$.
\end{thm}
Parabolic Sobolev embeddings of $\mathcal{H}_q^1(Q_T)$ into $L^{s}(Q_T)$ for $s>1$ satisfying $\frac{1}{s}=\frac{1}{q'}-\frac{1}{N+2}$ (cf. \cite[Appendix A]{cg20}) imply the following result:
\begin{cor}\label{corp}
Under the assumptions of Theorem \ref{mainberndual}, there exists a constant $C>0$ independent of $u$ such that
\[
\|\rho\|_{L^{p'}(Q_\tau)}\leq C(\left(\|\nabla u\|_{L^\infty(Q_\tau)}^{1-\theta}+1\right),\ p>\frac{(N+2)(\gamma-1)}{\gamma+1}.
\]
\end{cor}
The proof of Theorem \ref{mainberndual} follows the main line of \cite[Section 3.1]{cg20}. We first recall the following crucial representation formula, that easily follows by multiplying \eqref{hjm} by $-\rho$ and \eqref{fpms} by $u$:
\begin{prop}\label{repr}
Let $u$ be a solution to \eqref{hjm} and $\rho$ be a solution to \eqref{fpms}. Then the following identity holds
\[
\int_{\T^N}u(x,\tau)\rho_\tau(x)\,dx=\int_{\T^N}u_0(x)\rho(x,0)\,dx+\iint_{Q_\tau}L(x,H_p(x,\nabla u(x,t))\rho\,dxdt+\iint_{Q_\tau}f\rho\,dxdt.
\]
\end{prop} 
With the aid of Proposition \ref{repr} we first prove the following sup-norm estimate for solutions to \eqref{hjm}. This slightly extends \cite[Proposition 3.7]{cg20} and \cite[Theorem 2.3]{Goffi21JEE} to problems with mixed diffusion.
\begin{prop}\label{sup}
Let $f\in L^q(Q_T)$, $q>\frac{N+2}{2}$. Any solution to \eqref{hjm} satisfies
\[
\|u(\cdot,t)\|_{C(\T^N)}\leq C\ ,t\in[0,T],
\]
where $C$ depends on $T,N,q,s,\|f\|_{L^q(Q_T)}$.
\end{prop}
\begin{proof}
We first prove that
\begin{equation}\label{upper}
u(x,\tau)\leq \|u_0\|_{C(\T^N)}+C\|f\|_{L^q(\T^N)}\ ,\tau\in(0,T),\ x\in\T^N.
\end{equation}
Consider the strong nonnegative solution of the following problem
\[
\begin{cases}
\partial_t\mu-\Delta \mu+(-\Delta)^s \mu=0&\text{ in }Q_\tau,\\
\mu(x,\tau)=\mu_\tau(x)&\text{ in }\T^N\ .
\end{cases}
\]
where $\mu_\tau\in C^\infty$, $\mu_\tau\geq0$ and $\|\mu_\tau\|_1=1$. By Corollary \ref{corfpmain} with $b\equiv0$ and Sobolev embeddings, we have $\|\mu\|_{L^{q'}}\leq C$, $q'<\frac{N+2}{N}$. Using $\mu$ as a test function in \eqref{hjm}, one obtains
\[
\int_{\T^N}u(x,\tau)\mu_\tau(x)\,dx=\int_{\T^N}u_0(x)\mu(x,0)\,dx-\iint_{Q_\tau}H(x,\nabla u(x,t))\mu\,dxdt+\iint_{Q_\tau}f\mu\,dxdt.
\]
Then, the upper bound on $u$ follows by duality using the H\"older's inequality on the first and third integral, the fact that $\|\mu(t)\|_1=1$ for all $t\in(0,\tau)$ and also that $H,\mu\geq0$. The bound from below can be obtained in the same manner testing \eqref{hjm} against the solution of \eqref{fpms}, as in \cite[Proposition 3.7]{cg20}, and it is based on the representation formula in Proposition \ref{repr}.
\end{proof}
This bound directly leads to the following integrability estimate on the velocity field $b$ of \eqref{fpms} with respect to $\rho$ computed along the solution of the equation \eqref{hjm}. Its proof follows that of Proposition 3.2 in \cite{cg20}.
\begin{prop}\label{cross}
Let $u$ be a solution to \eqref{hjm} and $\rho$ be a solution to \eqref{fpms}. Then, there exists a constant $C>0$ depending on $q,N,T,s,C_H,\|f\|_{L^r(Q_T)}$, $r>\max\left\{\frac{N+2}{2},\frac{(N+2)(\gamma-1)}{\gamma}\right\}$ such that
\[
\iint_{Q_\tau}|\nabla u(x,t)|^{k}\rho\,dxdt\leq C\ ,1\leq k\leq \gamma.
\]
\end{prop}
\begin{rem}
Similar estimates as those in Proposition \ref{cross} can be found in \cite{cg20,CGpar,Goffi21JEE,Goffi21dir}.
\end{rem}
\begin{proof}
We rewrite the identity in Proposition \ref{repr} as
\[
\iint_{Q_\tau}L(x,H_p(x,\nabla u(x,t))\rho\,dxdt=\int_{\T^N}u(x,\tau)\rho_\tau(x)\,dx-\int_{\T^N}u_0(x)\rho(x,0)\,dx-\iint_{Q_\tau}f\rho\,dxdt.
\]
We pick $r>1$ such that
\[
\frac{(N+2)(\gamma-1)}{\gamma}<r<N+2<q.
\]
We use \eqref{L1}, the H\"older's inequality and the upper bound in \eqref{upper} to find
\begin{multline*}
C_L^{-1}\iint_{Q_\tau}|H_p(x,\nabla u(x,t)|^{\gamma'}\rho\,dxdt\leq 2\|u\|_{C(Q_T)}+\|f\|_{L^r(Q_T)}\|\rho\|_{L^{r'}(Q_T)}\\
\leq 2(\|u_0\|_\infty+\|f\|_{L^q(Q_\tau)})+\|f\|_{L^r(Q_\tau)}\|\rho\|_{L^{r'}(Q_\tau)}.
\end{multline*}
Let $\bar q$ be such that
\[
r'=\frac{(N+2)\bar q'}{N+2-\bar q'}.
\]
By the embedding $\mathcal{H}^1_{\bar q'}(Q_\tau)\hookrightarrow L^{r'}(Q_\tau)$ and choosing $r>\frac{N+2}{2}$ we have $\bar q'<\frac{N+2}{N+1}$. We are thus in position to apply Corollary \ref{corfpmain} and obtain
\[
\|\rho\|_{L^{r'}(Q_\tau)}\leq C(\|\rho\|_{\mathcal{H}^1_{\bar q'}(Q_\tau)}+1)\leq \tilde C\left(\iint_{Q_\tau}|H_p(x,\nabla u)|^{m'}\rho\,dxdt+1\right)\ ,
\]
where
\[
m'=1+\frac{N+2}{\bar q}.
\]
We thus end up with
\begin{multline*}
C_L^{-1}\iint_{Q_\tau}|H_p(x,\nabla u(x,t)|^{\gamma'}\rho\,dxdt\leq 2(\|u_0\|_\infty+\|f\|_{L^q(Q_\tau)})\\
+\tilde C\|f\|_{L^r(Q_\tau)}\left(\iint_{Q_\tau}|H_p(x,\nabla u)|^{m'}\rho\,dxdt+1\right).
\end{multline*}
The last integral can be absorbed on the left-hand side by the weighted Young's inequality since
\[
m'=1+\frac{N+2}{\bar q}=\frac{N+2}{r}<\gamma',
\]
so that we conclude the estimate.
\end{proof}
We are now ready for the proof of the main result of this section.
\begin{proof}[Proof of Theorem \ref{mainberndual}]
Since $q$ in \eqref{boundq} always satisfies $q'<\frac{N+2}{N+1}$, we can apply Corollary \ref{corfpmain} combined with the hypotheses \eqref{H} to conclude
\begin{multline*}
\|\rho\|_{\mathcal{H}_{\widetilde{q}'}^1(Q_\tau)}\leq C_1\left(\iint_{Q_\tau}|H_p(x,\nabla u(x,t))|^{m'}\rho\,dxdt+1\right)\leq C_2\left(\iint_{Q_\tau}|\nabla u|^{(\gamma-1)m'}\rho\,dxdt+1\right)\\
\leq C_3\left(\|\nabla u\|_{L^\infty(Q_\tau)}^{1-\theta}\iint_{Q_\tau}|\nabla u|^{(\gamma-1)m'-1+\theta}\rho\,dxdt+1\right),
\end{multline*}
where
\[
m'=1+\frac{N+2}{q}.
\]
We now choose $\theta>0$ small enough so that
\[
k=(\gamma-1)m'-1+\theta\leq \gamma.
\]
Note that this can be done by the initial choice of $q$. Thus, we use Proposition \ref{cross} and conclude the estimate.
\end{proof}
%We end this section with the following central bound on $\rho$:
%\begin{cor}
%Let $u$ be a solution to \eqref{hjm} and $\rho$ be a solution to \eqref{fpms}. Assume that $q$ satisfies \eqref{assq}. Then, there exists a positive constant $C$ depending on $N,q,T,s,C_H,\|f\|_q$
%\end{cor}
\section{Proof of the main result and further comments}\label{sec;proofs}
We start with the following refined variational Bochner identity that will allow us to exploit the information $f\in L^q_{x,t}$:
\begin{lemma}\label{bocdual}
Let $u$ be a classical solution to \eqref{hjm} and $\rho$ be a solution to \eqref{fpms}. Then, the following identity holds
\begin{multline}\label{bocfinal}
\int_{\T^N}w(x,\tau)\rho_\tau(x)\,dx+\iint_{Q_\tau}|D^2u(x,t)|^2\rho\,dxdt+\iint_{Q_\tau}\mathcal{I}[\nabla u](x,t)\rho\,dxdt\\
=-\iint_{Q_\tau}H_x(x,\nabla u(x,t))\cdot \nabla u(x,t)\rho\,dxdt\\
-\iint_{Q_\tau}f(x,t)\mathrm{div}(\nabla u(x,t)\rho)\,dxdt+\int_{\T^N}w(x,0)\rho(x,0)\,dx,
\end{multline}
where $\mathcal{I}[\nabla u](x,t)=\frac{1}{2}\int_{\T^N}|\nabla u(x,t)-\nabla u(x+y,t)|^2K(y)\,dy$.
\end{lemma}
\begin{proof}
It is sufficient to multiply \eqref{eqwid} by the adjoint variable $\rho$ solving \eqref{fpms}, integrate over the cylinder $Q_\tau$ and finally integrate by parts the terms involving the time derivative and the one involving the source term $f(x,t)$ of the equation.
\end{proof}
We are now ready for the proof of the main result
\begin{proof}[Proof of Theorem \ref{main1}]
We start with \eqref{bocfinal}, observing that the third term on the left-hand side is nonnegative since $\rho$ and the nonlocal term $\mathcal{I}[\nabla u]$ are nonnegative, cf \cite[Remark 20.3]{GarofaloFrac}. Thus, we end up with the inequality
\begin{multline*}
\int_{\T^N}w(x,\tau)\rho_\tau(x)\,dx+\iint_{Q_\tau}|D^2u(x,t)|^2\rho\,dxdt\leq -\iint_{Q_\tau}H_x(x,\nabla u(x,t))\cdot \nabla u(x,t)\rho\,dxdt\\
-\iint_{Q_\tau}f(x,t)\mathrm{div}(\nabla u(x,t)\rho)\,dxdt+\int_{\T^N}w(x,0)\rho(x,0)\,dx,
\end{multline*}
First, we observe that the chain rule applied to the second term of the right-hand side leads to
\[
-\iint_{Q_\tau}f\mathrm{div}(\nabla u\rho)\,dxdt=-\iint_{Q_\tau}f\Delta u\rho\,dxdt-\iint_{Q_\tau}f\nabla u\cdot \nabla\rho\,dxdt.
\]
We are now left to estimate all the terms appearing on the right-hand side. We start with the one involving $H_x$. Using \eqref{H}, Proposition \ref{cross} with $k=\gamma$ and the Young's inequality we conclude
\begin{multline*}
\iint_{Q_\tau}|H_x(x,\nabla u(x,t))||\nabla u(x,t)|\rho\,dxdt\leq C_H\|\nabla u\|_{L^\infty(Q_\tau)}\iint_{Q_\tau}|\nabla u|^\gamma\rho\,dxdt+C_H\tau\\
\leq C_1+\frac{1}{8}\|\nabla u\|_{L^\infty(Q_\tau)}^2.
\end{multline*}
We now estimate the terms involving the source of the equation. We use the Cauchy-Schwarz inequality \eqref{CD}, the Young's inequality together with the H\"older's inequality with an exponent $\tilde p>1$ to be later determined to find that
\begin{multline*}
\iint_{Q_\tau}|f||\Delta u|\rho\,dxdt\leq \sqrt{N}\iint_{Q_\tau}|f||D^2 u|\rho\,dxdt\leq \frac12\iint_{Q_\tau}|D^2 u|^2\rho\,dxdt+\frac{N}{2}\iint_{Q_\tau}f^2\rho\,dxdt\\
\leq \frac12\iint_{Q_\tau}|D^2 u|^2\rho\,dxdt+\frac{N}{2}\|f\|^2_{L^{2\tilde{p}}(Q_\tau)}\|\rho\|_{L^{\tilde{p}'}(Q_\tau)}.
\end{multline*}
The last term can be bounded by $\|f\|_{L^q(Q_\tau)}$ through Theorem \ref{mainberndual} by choosing $\tilde p$ satisfying
\[
\frac{2(N+2)(\gamma-1)}{\gamma+1}<2\tilde{p}\leq q.
\]
This can always be done in view of \eqref{assq}. Therefore, we have by Corollary \ref{corp} and the Young's inequality
\begin{multline*}
\iint_{Q_\tau}|f||\Delta u|\rho\,dxdt\leq \frac{N}{2}\|f\|_{L^{q}(Q_\tau)}^2(\|\nabla u\|_{L^\infty(Q_\tau)}^{1-\theta}+1)+\frac12\iint_{Q_\tau}|D^2 u|^2\rho\,dxdt\\
\leq C_2+\frac18\|\nabla u\|_{L^\infty(Q_\tau)}^2+\frac12\iint_{Q_\tau}|D^2 u|^2\rho\,dxdt.
\end{multline*}
Using the H\"older's inequality, Theorem \ref{mainberndual}, and the weighted Young's inequality we deduce
\begin{multline*}
\iint_{Q_\tau}|f||\nabla u||\nabla\rho|\,dxdt\leq \|\nabla u\|_{L^\infty(Q_\tau)}\|f\|_{L^q(Q_T)}\|\nabla \rho\|_{L^{q'}(Q_T)}\leq \|\nabla u\|_{L^\infty(Q_\tau)}\|f\|_{L^q(Q_T)}\| \rho\|_{\mathcal{H}^1_{q'}(Q_T)}\\
\leq C_3\|\nabla u\|_{L^\infty(Q_\tau)}\|f\|_{L^q(Q_T)}\left(\|\nabla u\|_{L^\infty(Q_\tau)}^{1-\theta}+1\right)\leq \frac18\|\nabla u\|_{L^\infty(Q_\tau)}^2+C_4.
\end{multline*}
Finally, using the conservation of the $L^1$ norm for the transport-diffusion equation we have $\|\rho(x,0)\|_{L^1(\T^N)}=1$, and we conclude
\[
\int_{\T^N}w(x,0)\rho(x,0)\,dx\leq \frac{1}{2}\|\nabla u_0\|_{L^\infty(\T^N)}^2.
\]
Plugging all the estimates in the first inequality we have
\begin{multline*}
\int_{\T^N}\frac12|\nabla u(x,\tau)|^2\rho_\tau(x)\,dx+\iint_{Q_\tau}|D^2u(x,t)|^2\rho\,dxdt\leq C_5+\frac{1}{2}\|\nabla u_0\|_{L^\infty(\T^N)}\\
+\frac38\|\nabla u\|_{L^\infty(Q_\tau)}^2+\frac12\iint_{Q_\tau}|D^2 u|^2\rho\,dxdt,
\end{multline*}
where $C_5=\max\{C_1,C_2,C_4\}$ depends on $N,C_H,q,T,s,\|f\|_q$. The last inequality holds for all smooth $\rho_\tau\geq0$ such that $\|\rho_\tau\|_1=1$, which gives by duality
\[
\frac12\|\nabla u(\cdot,\tau)\|_{L^\infty(\T^N)}^2\leq C_5+\frac{1}{2}\|\nabla u_0\|_{L^\infty(\T^N)}^2\\
+\frac38\|\nabla u\|_{L^\infty(Q_\tau)}^2.
\]
The estimate then follows by passing to the supremum over $\tau\in(0,T)$ on the left-hand side.
\end{proof}
%The well-posedness of the problem then follows by using a contraction mapping method combined with a continuation argument
%\begin{proof}[Proof of Theorem \ref{main2}]
%The proof consists of two steps. First, one proves the small-time existence via a contraction mapping argument on the space
%\[
%\mathcal{S}:=\{u\in W^{2,1}_q(Q_T): u(0)=u_0, \|u\|_{W^{2,1}_q(Q_T)}\leq a\ ,q>N+2\}
%\]
%for $a$ chosen large enough. This can be done through the same steps of \cite[Proposition 3.11]{CGsima}, see also Lemma \ref{czmix}. Owing to the gradient bound from Theorem \ref{main1}, one can proceed through a continuation argument in the space $W^{2,1}_q$, as it is done in Step 2 of \cite[Proposition 3.11]{CGsima}. Higher regularity follows from parabolic Sobolev embeddings for $q>N+2$, see e.g. \cite{CGsima}.
%\end{proof}
Some final remarks are in order:
\begin{rem}\label{test}
The argument continues to be valid for functions $u\in W^{2,1}_q$ instead of smooth solutions, without need to assume $f\in C^1$. This can be done using the test function $\varphi=-\mathrm{div}(\nabla u\rho)$ in the variational formulation of \eqref{hjm}, as in \cite{CGell}.
\end{rem}
\begin{rem}\label{rem;higher}
When $f$ is more regular, the proof works even when $\eps=0$, i.e. for the equation driven by the sole fractional Laplacian, in the subcritical case $s\in\left(\frac12,1\right)$ and for strong solutions belonging to the space $\mathcal{H}_q^{2s}(Q_T)=\{\partial_t u,(I-\Delta)^su\in L^q(Q_T)\}$, as introduced in \cite{CGsima}. The procedure simplifies since one can avoid an integration by parts in \eqref{bocfinal} and use the identity
\begin{multline*}
\int_{\T^N}w(x,\tau)\rho_\tau(x)\,dx\leq \int_{\T^N}w(x,\tau)\rho_\tau(x)\,dx+\iint_{Q_\tau}\mathcal{I}[\nabla u](x,t)\rho\,dxdt\\
=-\iint_{Q_\tau}H_x(x,\nabla u(x,t))\cdot \nabla u(x,t)\rho\,dxdt
-\iint_{Q_\tau}\nabla f(x,t)\cdot \nabla u(x,t)\rho\,dxdt+\int_{\T^N}w(x,0)\rho(x,0)\,dx.
\end{multline*}
Then, if $f\in L^\infty(0,T;W^{1,\infty}(\T^N))$, we have
\[
\iint_{Q_\tau}\nabla f(x,t)\cdot \nabla u(x,t)\rho\,dxdt\leq \|f\|_{L^\infty_t(W^{1,\infty}_x)}\|\nabla u\|_{L^\infty(Q_\tau)}\tau,
\]
and then it is sufficient to use the Young's inequality. If, instead, $f\in L^q(0,T;W^{1,q}(\T^N))$ for $q$ such that
\[
q>\max\left\{N+2s,\frac{(N+2s)(\gamma-1)}{2s-1}\right\},
\] 
it is enough to argue as follows:
\[
\iint_{Q_\tau}\nabla f(x,t)\cdot \nabla u(x,t)\rho\,dxdt\leq \|\nabla u\|_{L^\infty(Q_\tau)}\iint_{Q_\tau}|\nabla f|\rho\,dxdt\leq \|\nabla u\|_{L^\infty(Q_\tau)}\|f\|_{L^p_t(W^{1,p}_x)}\|\rho\|_{L^{p'}(Q_\tau)}
\]
and use the estimate in Corollary 5.9 of \cite{Goffi21JEE} combined with the Young's inequality to conclude the gradient bound. We emphasize that an assumption like $f\in L^q(0,T;W^{1,q}(\T^N))$ lies in between those assumed for the use of the weak Bernstein method and the ones to run the Ishii-Lions argument. It remains an open problem whether the $L^\infty$ gradient bound holds for $f\in L^q_{x,t}$ and $\gamma>2s$: by scaling, we expect this can be true provided that $q>\frac{N+2s}{2s-1}$, $s\in(1/2,1)$.
\end{rem}
\begin{rem}
Lipschitz estimates in the subquadratic case can be obtained in a slightly different manner, combining Lemma \ref{czmix} with Gagliardo-Nirenberg interpolation inequalities. Indeed, regarding \eqref{hjm} as 
\[
\partial_tu-\Delta u+(-\Delta)^s u=-H(x,\nabla u)+f(x,t)
\]
one has
\[
\|u\|_{W^{2,1}_q}\lesssim \|\nabla u\|_{L^{q\gamma}}^\gamma+\|f\|_q+\|u_0\|_{W^{2-2/q,q}}.
\]
Then, Gagliardo-Nirenberg interpolation inequalities lead to
\[
\|\nabla u\|_{L^{q\gamma}}^\gamma\leq C\|\nabla u\|_{L^{2q}}^\gamma \lesssim \|u\|_{W^{2,1}_q}^{\frac{\gamma}{2}} \|u\|_{\infty}^{\frac{\gamma}{2}}.
\]
Since $\gamma\in(1,2)$, one can use the Young's inequality and the sup-norm estimate in Proposition \ref{sup} to conclude a bound on $\|u\|_{W^{2,1}_q}$ for $q>\frac{N+2}{2}$. By parabolic Sobolev embeddings it then follows that $\|\nabla u\|_{L^\infty}$ is bounded whenever $q>N+2$. The drawback of this approach is the requirement on $\gamma$ (that must be of subquadratic growth) and on the initial data, which needs to be more regular respect to Theorem \ref{main1}. The same idea works even for problems driven by the sole fractional Laplacian using interpolation estimates and the $L^\infty$ bounds in Theorem 2.3 of \cite{Goffi21JEE}, provided that $\gamma<2s$, $s\in(1/2,1)$. Still, the same approach of this manuscript can be refined to obtain maximal regularity estimates via new H\"older and $L^p$ bounds for mixed diffusion problems, see \cite{CGpar}. 
\end{rem}
\begin{rem}
We believe that Lipschitz bounds can be produced under the weaker restriction $q>N+2$ when $\gamma>3$, and this would be compatible with the elliptic results in \cite{Lions85}, but this remains an open problem even for viscous problems driven by the Laplacian.
\end{rem}

\end{document}